\documentclass[11pt,reqno]{amsart}
\usepackage{diagrams}
\usepackage{amssymb}
\usepackage[pagewise]{lineno}
\numberwithin{equation}{section}
\usepackage{cite}

\theoremstyle{definition}
\newtheorem{thm}{Theorem}[section]

\newtheorem{lem}[thm]{Lemma}

\newtheorem{prop}[thm]{Proposition}

\newtheorem{rem}[thm]{Remark}

\newtheorem*{conj}{Conjecture}

\DeclareMathOperator{\codim}{\mathrm{codim}}
\DeclareMathOperator{\NL}{\mathrm{NL}}

\DeclareMathOperator{\red}{\mathrm{red}}

\DeclareMathOperator{\Ima}{\mathrm{Im}}
\DeclareMathOperator{\Hilb}{\mathrm{Hilb}}

\DeclareMathOperator{\op}{\mathcal{O}_{\mathbb{P}^3}(d)}

\DeclareMathOperator{\p3}{\mathbb{P}^3}

\DeclareMathOperator{\pr}{\mathrm{pr}}

\DeclareMathOperator{\N}{\mathcal{N}}
\DeclareMathOperator{\T}{\mathcal{T}}

\DeclareMathOperator{\I}{\mathcal{I}}
\DeclareMathOperator{\mo}{\mathcal{O}}

\newcommand{\mr}[1]{\mathrm{#1}}
\newcommand{\mb}[1]{\mathbb{#1}}

\newcommand{\mc}[1]{\mathcal{#1}}
\newcommand{\ov}[1]{\overline{#1}}

\begin{document}

\title{On a conjecture of Harris}

\author[A. Dan]{Ananyo Dan}

\address{School of Mathematics and Statistics, University of Sheffield, Hicks building, Hounsfield Road, S3 7RH, UK}

\email{a.dan@sheffield.ac.uk}

\subjclass[2010]{$14$D$07$, $14$C$05$, $14$C$20$}

\keywords{Harris conjecture, Noether-Lefschetz locus, Hodge locus, flag Hilbert schemes}

\date{\today}

\begin{abstract}
 For $d \ge 4$, the Noether-Lefschetz locus $\NL_d$ parametrizes smooth, degree $d$ 
  surfaces in $\p3$ with Picard number at least $2$. 
 A conjecture of Harris states that there 
 are only finitely many irreducible components of the Noether-Lefschetz locus of non-maximal codimension.
 Voisin showed that the conjecture is false for sufficiently large $d$, but is true for $d \le 5$.
 She also showed that for $d=6, 7$, there are finitely many \emph{reduced}, irreducible components
 of $\NL_d$ of non-maximal codimension.
 In this article, we prove that for any $d \ge 6$, there are infinitely many \emph{non-reduced}
 irreducible components of $\NL_d$ of non-maximal codimension.
 \end{abstract}

\maketitle

\section{Introduction}
The underlying field is $\mathbb{C}$. By surface we will always mean a projective surface in $\p3$.
A classical result in the theory of surfaces, stated by M. Noether and later proved by Lefschetz, says that
for any $d \ge 4$, a very general, smooth, degree $d$ surface in $\p3$ is of Picard number $1$ 
(by \emph{Picard number} we mean the rank of the N\'{e}ron-Severi group).
Here, \emph{very general} means that the points on the parametrizing space $\mathbb{P}(H^0(\mathbb{P}^3, \op))$ of 
degree $d$ surfaces in $\p3$, corresponding to such surfaces,
lie outside a countable union of proper, closed subsets of $\mathbb{P}(H^0(\mathbb{P}^3, \op))$. 
The \emph{Noether-Lefschetz locus}, denoted $\NL_d$, is then defined to be the locus of 
smooth, degree $d$ surfaces in $\p3$ with Picard number at least $2$. 
The irreducible components of the Noether-Lefschetz locus have a natural (analytic) scheme structure, which we will now describe.

Denote by $U_d \subseteq \mathbb{P}(H^0(\mathbb{P}^3, \op))$ 
the open subscheme parametrizing smooth projective hypersurfaces in $\mathbb{P}^3$ of degree $d$.
Let \[\pi: \mathcal{X} \to U_d\] be the corresponding universal family. For a given $u \in U_d$, denote by $\mc{X}_u:=\pi^{-1}(u)$. 
Denote by \[\mb{H}:=R^2\pi_*\mathbb{Z} \mbox{ and } \mc{H}:=R^2\pi_*\mathbb{Z} \otimes \mo_{U_d}.\]
Using the Ehresmann's theorem, it is easy to check that $\mb{H}$ is a local system, hence $\mc{H}$ is a vector bundle. 
Fix a point $o \in \NL_d$ and $U \subseteq U_d$ a simply connected open neighbourhood of $o$ in $U_d$ (under the analytic topology).
It is easy to check that the restriction of $\mb{H}$ to $U$ is trivial and any class 
$\gamma_0 \in H^2(\mc{X}_o,\mathbb{Z})$ defines by flat transform, a section $\gamma \in \Gamma(U,\mb{H})$.
Let $\overline{\gamma}$ be the image
of $\gamma$ in $\mathcal{H}/F^1\mathcal{H}$, where 
$F^1\mc{H} \subset \mc{H}$ is a vector subbundle such that
for every $u \in U$, the fiber $(F^1\mc{H})_u \subset \mc{H}_u$ can be identified with $F^1 H^2(\mc{X}_u,\mb{C}) \subset
H^2(\mc{X}_u,\mb{C})$ (see \cite[\S $10.2.1$]{v4}).
If $\gamma_0$ belongs to $H^{1,1}(\mc{X}_o,\mb{C})$ i.e., $\gamma_0$ is a Hodge class, then 
the \emph{Hodge locus associated to} $\gamma_0$, denoted $\NL(\gamma_0)$, is defined as 
\[ \NL(\gamma_0):=\{u \in U | \overline{\gamma}(u)=0\},\]
where $\overline{\gamma}(u)$ denotes the value at $u$ of the section $\overline{\gamma}$.
The Hodge locus is equipped with a natural scheme structure (see \cite[\S $5.3.1$]{v5}).
The intersection of $\NL_d$ with $U$ is the union of $\NL(\gamma_0)$ as $\gamma_0$ ranges over the 
Hodge classes of $\mc{X}_u$ for all $u \in U$, such that the Hodge class is not a multiple of $c_1(\mo_{\mc{X}_u}(1))$ (see \cite[\S $5.3.3$]{v5}).
We say that the closure $\ov{\NL(\gamma_0)}$ of $\NL(\gamma_0)$ (in the Zariski topology), is 
an \emph{irreducible component} of $\NL_d$, if the underlying topological space is 
irreducible and as an (analytic) \emph{scheme} is not contained properly in any irreducible component of 
$\ov{\NL(\gamma')}$ for some Hodge class $\gamma'$ over $\mc{X}_o$, where $\gamma'$ is not a multiple of $c_1(\mo_{\mc{X}_o}(1))$. 
Two irreducible components 
of $\NL_d$ are isomorphic if they are isomorphic as analytic schemes (scheme structure as Hodge loci).
It was shown by Ciliberto-Harris-Miranda \cite{ca1} that 
any irreducible component $L$ of $\NL_d$ satisfies the inequality:
\[d-3 \le \codim(L,U_d) \le \binom{d-1}{3}.\]
If $\codim(L,U_d) = \binom{d-1}{3}$, then $L$ is called a \emph{general component}. Otherwise, $L$ is called a \emph{special component}.
Harris conjectures the following on the special components of the Noether-Lefschetz locus (see \cite{voicontr, M3}):

\begin{conj}[Harris]
Fix an integer $d \ge 5$. Then,
\begin{enumerate}
 \item{\bf{Topological Harris conjecture}:} Ignoring the natural analytic scheme structure (as Hodge locus)
 on the irreducible components of $\NL_d$, there are finitely many topological, special components of $\NL_d$. 
 
 \item{\bf{Analytic Harris conjecture}:} The Noether-Lefschetz locus $\NL_d$ contains finitely many special, irreducible components (by irreducible
 component we mean as above).
\end{enumerate}
\end{conj}

For $d \le 5$ the conjectures hold true (see \cite[Theorem $0.2$]{v3}).
The conjectures have been shown to be false by Voisin \cite{voicontr} for sufficiently large $d$.
For $d=6,7$, Voisin proved in \cite{voilieu} that $\NL_d$ has finitely many \emph{reduced}, 
special components.  
But there are several questions that are still open:
What is the largest $d'$ such that the Harris conjectures hold true for all $d \le d'$? For those $d$ for which the Harris conjectures fail, what is the 
largest $k$ such that there are finitely many special components of codimension at most $k$ and infinitely many special components of codimension
strictly greater than $k$? In this article we give a complete answer to 
these questions for the analytic Harris conjecture. We show:

\begin{thm}[see Theorem \ref{th2}]\label{th1}
 Let $d \ge 6$ and $X$ be a general smooth degree $d$ surface containing two coplanar lines, say $L_1, L_2$.
 Denote by $\gamma_a:=[L_1]+a[L_2] \in H^2(X,\mb{Z}) \cap H^{1,1}(X,\mb{C})$, for any $a \in \mb{Z}$.
 Then, for any $a, b \ge 0$ with $a \not= b$, we have $\NL(\gamma_a) \not= \NL(\gamma_b)$. Moreover, 
 \[\codim \NL(\gamma_0)=d-3,\, \codim \NL(\gamma_1)=2d-7 \mbox{ and } \codim \NL(\gamma_a)=2d-6 \mbox{ for } a \ge 2.\]
 In particular, we have following table which gives the number of irreducible components of $\NL_d$ with the given codimension:
 \[\begin{tabular}{|c|c|}
    \hline
    \mbox{codimension}&\mbox{number of irr. components}\\
    \hline
    $<d-3$&0\\
    d-3&1\\
    2d-7&1\\
    {$\leq 2d-7$} & 2\\
    2d-6&$\infty$\\
    \hline
   \end{tabular}\]
\end{thm}

 The theorem immediately disproves the analytic Harris conjecture for $d \ge 6$.
 Recall, Voisin in \cite{voicontr}  uses the existence of infinitely many general components of $\NL_4$
 to produce infinitely many sufficiently high degree special components (after replacing the coordinates with general high degree polynomials).
 In particular, the topological space underlying the special components are distinct.
 Her counterexample relies on a numerical inequality that holds \emph{only} for sufficiently large $d$ (see \cite[p. $686$]{voicontr}). 
 In contrast, we simply study the scheme structure of 
 the Hodge locus corresponding to different linear combinations 
 of coplanar lines. In particular, we show that 
 the 
 space parametrizing smooth, degree $d$ surfaces containing $2$ coplanar lines can be 
 equipped with infinitely many (distinct) scheme structures \emph{naturally} arising as the Hodge loci 
 associated to different combinations of the two coplanar lines (see Theorem \ref{th2}).
 As a result, the 
 infinite number of special components in this article have the same underlying topological space (but different analytic scheme structures),
 thus giving us an entirely different set of counterexamples from those in 
 \cite{voicontr}. The topological Harris conjecture is still open for small values of $d$.
 In a recent preprint, Movasati in \cite{hoss} uses computer calculations for explicit values of $d$ to give  a
description of possible counterexamples to the topological Harris conjecture.

 \vspace{0.2 cm}
\emph{Acknowledgements}:
The author is currently supported by ERCEA Consolidator Grant $615655$-NMST, by the Basque Government through the BERC $2018-2021$ program and 
Gobierno Vasco Grant $\mbox{IT}1094-16$, by the Spanish Ministry of Science, Innovation and Universities: BCAM Severo Ochoa accreditation SEV-$2017-0718$. 

\section{Proof of main theorem}
In this section we prove Theorem \ref{th1}.
Fix an integer $d \ge 6$. 

%

\subsection{Cohomology computations of the invertible sheaves associated to lines} 
To prove Theorem \ref{th1} we need the following basic computation on the first and second cohomology group of the invertible sheaf
associated to a line contained in a smooth, degree $d$ surface in $\p3$.

\begin{lem}\label{lem1}
Let $X$ be a smooth, degree $d$ surface
containing two coplanar lines, say $L, L'$. Then, 
\[H^0(\mo_X(L'))=H^0(\mo_X(L))=\mb{C}=H^0(\mo_X(L \cup L')) \mbox{ and } H^1(\mo_X(L))=0=H^1(\mo_X(L')).\]
\end{lem}

\begin{proof}
By the adjunction formula, $L^2=-2-(d-4)=2-d$, 
 which is less than zero for $d \ge 6$. Hence, $H^0(\N_{L|X})=0$.
 Using the short exact sequence 
 \begin{equation}\label{eq01}
 0 \to \mo_X \to \mo_X(L) \to \N_{L|X} \to 0
\end{equation}
we conclude that $\mb{C}=H^0(\mo_X)=H^0(\mo_X(L))$. Similarly, we can show that $\mb{C}=H^0(\mo_X(L'))$.
Next, we consider the short exact sequence:
\begin{equation}\label{eq04}
 0 \to \mo_X(L) \to \mo_X(L \cup L') \to \mo_{L'} \otimes \mo_X(L \cup L') \to 0
\end{equation}
obtained by tensoring with $\mo_X(L \cup L')$ the short exact sequence:
\[0 \to \mo_X(-L') \to \mo_X \to \mo_{L'} \to 0.\]
Note that, $L'.(L+L')=1+(2-d)<0$ for $d \ge 6$.
This implies, $H^0( \mo_{L'} \otimes \mo_X(L \cup L'))=0$. 
By \eqref{eq04}, we then have $H^0(\mo_X(L \cup L'))=H^0(\mo_X(L))=\mb{C}$.
This proves the first part of the lemma. We now show that 
$H^1(\mo_X(L))=0=H^1(\mo_X(L'))$.

Recall, the Castelnuovo-Mumford regularity of $\I_L$ is one (see \cite[Example $1.8.2$]{lazaI}).
This implies that $H^i(\I_L(j))=0$ for $i \ge 1$ and $j \ge 0$.
Consider now the short exact sequence:
\begin{equation}\label{eq02}
 0 \to \I_X(d-4) \to \I_L(d-4) \to \mo_X(-L)(d-4) \to 0.
\end{equation}
 As $\I_X \cong \mo_{\p3}(-d)$, we have $H^2(\I_X(d-4))=0$. Using the long exact sequence associated to \eqref{eq02}, 
 we conclude that $H^1(\mo_X(-L)(d-4))=0$.
 By Serre duality this implies \[H^1(\mo_X(L))=H^1(\mo_X(-L)(d-4))=0.\] 
Similarly, we can show that $H^1(\mo_X(L'))=0$. This proves the lemma.
\end{proof}

\subsection{Flag Hilbert schemes}\label{flagsec}

Let $X$ be a smooth, degree $d$ surface in $\p3$ containing two (distinct) coplanar lines, say $L_1$ and $L_2$.
Denote by $P_0$ (resp. $P_1, Q_d$) the Hilbert polynomial of $L_1$ (resp. $L_1 \cup L_2, X$).
Denote by $\mr{Hilb}_{P_0,P_1,Q_d}$ the flag Hilbert scheme parametrizing triples $(Z_1 \subset Z_2 \subset Z_3)$
with $Z_1$ (resp. $Z_2, Z_3$) having Hilbert polynomial $P_0$ (resp. $P_1, Q_d$). Denote by $\mr{Hilb}_{P_i}$ and $\mr{Hilb}_{Q_d}$ the 
Hilbert scheme associated to the Hilbert polynomial $P_i$ and $Q_d$, respectively for $i=0, 1$. See \cite[\S $4.3, 4.5$]{S1}
for a detailed discussion on (flag) Hilbert schemes.

\begin{prop}\label{lem4}
 The flag Hilbert scheme $\mr{Hilb}_{P_0, P_1, Q_d}$ is  reduced. In particular, the scheme-theoretic image under the natural projection map
 \[\pr: \mr{Hilb}_{P_0, P_1, Q_d} \to \mr{Hilb}_{Q_d}\]
 is reduced.
\end{prop}

\begin{proof}
 Consider the natural projection map:
 \[\pr_0: \mr{Hilb}_{P_0, P_1, Q_d} \to \mr{Hilb}_{P_0}.\]
 Clearly, this map is surjective as for every line $L$ there exists infinitely many lines $L'$ lying on the same plane as $L$
 such that $L \cup L'$ is contained in a smooth, degree $d$ surface in $\p3$.
 Note that for any point $t \in \mr{Hilb}_{P_0}$, we have 
 \[\dim T_t\mr{Hilb}_{P_0} = h^0(\N_{L_t|\p3})=h^0(\mo_{L_t}(1))+h^0(\mo_{L_t}(1))=4=\dim \Hilb_{P_0},\]
 where $L_t$ is the line corresponding to the point $t$. Hence, $\mr{Hilb}_{P_0}$ is smooth. 
 A standard exercise in commutative algebra tells us that given a morphism of schemes, the domain is reduced if  
 the scheme-theoretic image and every fiber is  reduced. Therefore, it is sufficient 
 to check that the every fiber to the morphism $\pr_0$ is  reduced.
 
 Note that, the morphism $\pr_0$ factors through $\mr{Hilb}_{P_0,P_1}$.
 Denote by \[\pr_1:\mr{Hilb}_{P_0,P_1} \to \mr{Hilb}_{P_0} \mbox{ and } \pr_2:\mr{Hilb}_{P_0,P_1,Q_d} \to \mr{Hilb}_{P_0,P_1}\]
 the natural projections. Since every conic can be embedded in a degree $d$ surface in $\p3$, the scheme-theoretic image of $\pr_0^{-1}(t)$
 under the morphism $\pr_2$ coincides with the fiber $\pr_1^{-1}(t)$. The dimension of $\pr_1^{-1}(t)$ equals 
 $\dim \mb{P}(H^0(\mo_{\mb{P}^2}(1)))+1$, where the first term is the dimension of the space of lines contained in 
 the same plane as $L_t$ (after fixing the plane) and the second term is the dimension of the space of planes in $\p3$ containing 
 $L_t$. For any line $L'$ contained in the same plane as $L_t$, we have 
 \[h^0(\mo_{L_t+L'}(1))=h^0(\mo_{L_t}(1))+1 \mbox{ and } h^0(\mo_{L_t+L'}(2))=2.2+1=5\]
 where the last equality follows from the fact that the Castelnuovo-Mumford regularity of $\mo_{L_t+L'}$ is one, 
 which implies $h^0(\mo_{L_t+L'}(2))$ equals $P_{L_t+L'}(2)$ for the Hilbert polynomial $P_{L_t+L'}(n)$ of $L_t+L'$.
 Similarly, $h^0(\mo_{L_t}(2))=2+1=3$.
 Since $\N_{L_t+L'|\p3} \cong \mo_{L_t+L'}(1) \oplus \mo_{L_t+L'}(2)$,  \cite[Ex. II.$8.4$]{R1} implies that the restriction morphism 
 \[\rho:H^0(\N_{L_t+L'|\p3}) \to H^0(\N_{L_t+L'|\p3} \otimes_{\mo_{\p3}} \mo_{L_t})\]
 is surjective. Hence, \[\dim \ker(\rho)=h^0(\mo_{L_t+L'}(1))+h^0(\mo_{L_t+L'}(2))-h^0(\mo_{L_t}(1))-h^0(\mo_{L_t}(2))=1+5-3=3.\]
Using \cite[Remarks $4.5.4$]{S1}, we have 
$\dim T_{(L_t \subset L_t+L')}\pr_1^{-1}(t) =\ker \rho$, which by our computation equals $\dim \pr_1^{-1}(t)$. Hence, $\pr_1^{-1}(t)$ is reduced.

  The fiber over the point corresponding to the pair  $(L_t \subset L_t \cup L')$ for the composed morphism 
  \[\pr_0^{-1}(t) \hookrightarrow \mr{Hilb}_{P_0,P_1,Q_d} \xrightarrow{\pr_2} \mr{Hilb}_{P_0,P_1}\]
    is isomorphic to $\mb{P}(H^0(\I_{L_t \cup L'}(d)))$, which is reduced. 
    Since  $\pr_2(\pr_0^{-1}(t))=\pr_1^{-1}(t)$ is reduced, this implies that $\pr_0^{-1}(t)$ is reduced.
 Hence, $\mr{Hilb}_{P_0, P_1, Q_d}$ is reduced. The second part of the lemma is direct (scheme-theoretic image of a reduced scheme is reduced).
 This proves the proposition.
\end{proof}

\subsection{Proof of Theorem \ref{th1}}
Let $X$ be a general, smooth, degree $d$ surface in $\p3$ containing $2$ distinct coplanar lines, say $L_1, L_2$.
We use the notations as in \S \ref{flagsec}.
Denote by $o \in U_d$ the point corresponding to $X$.
Let $U \subset U_d$ be a simply connected neighbourhood of $o$ in $U_d$ as before.

\begin{prop}\label{prop01}
The (Zariski) closure of $\NL([L_1]) \cap \NL([L_2])$ in $U_d$ is isomorphic to the scheme-theoretic image of 
the morphism $\pr$ as in Proposition \ref{lem4}, intersected with $U_d$.
\end{prop}

 \begin{proof}
 Denote by $W:=\Ima(\pr) \cap U_d$, where $\pr$ is as in Proposition \ref{lem4}.
 Note that, $W$ parametrizes smooth, degree $d$ surfaces in $\p3$ containing two coplanar lines. Hence,
 $\ov{\NL([L_1]) \cap \NL([L_2])}$ contains $W$. We now prove the reverse inclusion i.e., $\ov{\NL([L_1]) \cap \NL([L_2])} \subset W$.
 Denote by 
 \[\pi': \mc{X}' \to \ov{\NL([L_1]) \cap \NL([L_2])}\] the restriction of 
  $\pi$ to $\ov{\NL([L_1]) \cap \NL([L_2])}$. 
 By Lefschetz $(1,1)$-theorem, there exist invertible sheaves $\mc{L}_1$ and $\mc{L}_2$ over $\mc{X}'$ such that $\mc{L}_1|_X \cong \mo_X(L_1)$
 and $\mc{L}_2|_X \cong \mo_X(L_2)$. 
 Using Lemma \ref{lem1} and the upper semi-continuity of cohomology (see \cite[Theorem III.$12.8$]{R1}), 
 there exists an open neighbourhood $V \subset \ov{\NL([L_1]) \cap \NL([L_2])}$ of $o$ such that 
for all $v \in V$, we have 
 \[h^0(\mc{L}_{1,v})= 1 = h^0(\mc{L}_{2,v}) \mbox{ and } h^1(\mc{L}_{1,v})= 0 = h^1(\mc{L}_{2,v}), \mbox{ where } \mc{L}_{1,v}:=\mc{L}_1|_{\mc{X}_v}
 \mbox{ and } \mc{L}_{2,v}:=\mc{L}_2|_{\mc{X}_v}.\]
By \cite[Theorem III.$12.11$]{R1}, for every $v \in V$, the natural morphisms
\[\pi'_*\mc{L}_1 \otimes k(v) \to H^0(\mc{L}_{1,v}) \mbox{ and } \pi'_*\mc{L}_2 \otimes k(v) \to H^0(\mc{L}_{2,v})\]
are isomorphisms. Hence, after contracting $V$ if necessary, there exist sections $s_1 \in \Gamma(V, \pi'_*\mc{L}_1)$
and $s_2 \in \Gamma(V,\pi'_*\mc{L}_2)$ such that its image $s_{1,v}$ and $s_{2,v}$ in $H^0(\mc{L}_{1,v})$ and $H^0(\mc{L}_{2,v})$, respectively
are non-zero for all $v \in V$. The sections $s_1$ and $s_2$ give rise to the short exact sequence:
\[0 \to \mc{L}_i^\vee|_V \xrightarrow{.s_i} \mo_{\mc{X}_V} \to \mo_{Z(s_i)} \to 0,\]
where $Z(s_i)$ is the \emph{zero locus of the section } $s_i$ in $\mc{X}_V:=\pi^{-1}(V)$, for $i=1,2$.
Since $s_{i,v}$ is non-zero, the natural morphism 
\[\mc{L}_{i,v}^\vee \xrightarrow{.s_{i,v}} \mo_{\mc{X}_v}\]
is injective for $i=1, 2$. By the local criterion of flatness (see \cite[p. $150$, $(20.\mbox{E})$]{mat2}),
we conclude that $Z(s_i)$ is flat over $V$ for $i=1, 2$. Denote by 
$\mc{L}^{\mr{eff}}_1:=Z(s_i)$ and $\mc{L}^{\mr{eff}}_2:=Z(s_2)$.
It is easy to check that the effective divisor 
$\mc{L}^{\mr{eff}}_1+\mc{L}^{\mr{eff}}_2$ of $\mc{X}_V$ is also flat over $V$.
By the universal property, of the flag Hilbert schemes (see \cite[Theorem $4.5.1$]{S1}), the triple 
\[(\mc{L}^{\mr{eff}}_1 \subset \mc{L}^{\mr{eff}}_1+\mc{L}^{\mr{eff}}_2 \subset \mc{X}_V)\]
induces a morphism from $V$ to $\mr{Hilb}_{P_0, P_1, Q_d}$ such that the composition 
\[V \to \mr{Hilb}_{P_0, P_1, Q_d} \xrightarrow{\pr} \mr{Hilb}_{Q_d}\]
is the natural inclusion. This implies, a dense open subscheme of $\ov{\NL([L_1]) \cap \NL([L_2])}$ lies in the scheme-theoretic image of $\pr$.
Since the morphism $\pr$ is proper, we conclude that $\ov{\NL([L_1]) \cap \NL([L_2])}$ lies in the scheme-theoretic image of $\pr$.
So, we have the reverse inclusion. Hence. $\ov{\NL([L_1]) \cap \NL([L_2])}=\Ima(\pr)$. This proves the proposition.
 \end{proof}
%
%

\begin{thm}\label{th2}
For any $a \in \mb{Z}$, denote by $\gamma_a:=[L_1]+a[L_2] \in H^2(X,\mb{Z}) \cap H^{1,1}(X,\mb{C})$.
 Then, for any $a, b \ge 0$ with $a \not= b$, we have $\NL(\gamma_a) \not= \NL(\gamma_b)$. Moreover, 
 \[\codim \NL(\gamma_0)=d-3,\, \codim \NL(\gamma_1)=2d-7 \mbox{ and } \codim \NL(\gamma_a)=2d-6 \mbox{ for } a \ge 2.\]
\end{thm}

\begin{proof}
Denote by $\ov{\NL(\gamma_a)}$ the closure of $\NL(\gamma_a)$ in $U_d$ under the Zariski topology.
 By \cite[Theorem $0.2$]{v3} $\ov{\NL(\gamma_0)}$ (resp. $\ov{\NL(\gamma_1)}$) parametrizes smooth, degree $d$ surfaces
 containing a line (resp. a conic) and is of codimension $d-3$ (resp. $2d-7$) in $U_d$. Furthermore, both $\NL(\gamma_0)$ and 
 $\NL(\gamma_1)$ are reduced. 
 We first note that for $a \ge 2$, we have $\NL(\gamma_a) \not= \NL(\gamma_0)$,  $\NL(\gamma_a) \not= \NL(\gamma_1)$ 
 and $\codim \NL(\gamma_a)=2d-6$. Indeed, $\ov{\NL(\gamma_a)}$ contains the space $W$ parametrizing smooth, degree $d$ surfaces 
 containing $2$ coplanar lines. It is easy to compute that $\codim W=2d-6$. Hence,
 $\codim \NL(\gamma_a) \le 2d-6$. By \cite[Theorem $0.2$]{v3}, either $\NL(\gamma_a)=\NL(\gamma_0)$ or $\NL(\gamma_a)=\NL(\gamma_1)$ or 
 $\codim \NL(\gamma_a)=2d-6$. If $\NL(\gamma_a)=\NL(\gamma_0)$ or $\NL(\gamma_a)=\NL(\gamma_1)$, then $\ov{\NL(\gamma_a)}$ parametrizes
 smooth, degree $d$ surfaces such that both $[L_1]$ and $[L_2]$ remains a Hodge class, in particular 
 \[\NL(\gamma_a)=\NL([L_1]) \cap \NL([L_2]).\] 
 But, Proposition \ref{prop01} then implies that $\ov{\NL(\gamma_a)}=\Ima(\pr)=W$, which is of codimension $2d-6$. This gives us a contradiction.
 Hence, for $a \ge 2$, $\NL(\gamma_a) \not= \NL(\gamma_0)$,  $\NL(\gamma_a) \not= \NL(\gamma_1)$ 
 and $\codim \ov{\NL(\gamma_a)}=2d-6$.
  
 We now show that for $a \ge 2$ and $u \in \NL(\gamma_a)$ general, we have $\codim T_u\NL(\gamma_a) \le 2d-7$.
 Indeed, using \cite[$(4.a.4)$]{GH}, we have for all $a \ge 2$,
 \[H^0(K_X(-L_1-L_2)) \subset H^{2,0}(-\gamma_a):=\{\psi \in H^{2,0}(X,\mb{C})| (t \cup \psi) \cup \gamma_a=0 \mbox{ for all } t \in H^1(\T_X)\}.\]
 Recall from \cite[p. $211$]{GH} that 
 \[\codim T_o\NL(\gamma_a)=\dim H^{2,0}(X,\mb{C})-\dim H^{2,0}(-\gamma_a)\]
 which by our calculations is bounded above by $h^0(K_X)-h^0(K_X(-L_1-L_2))$.
 Consider now the short exact sequence:
 \begin{equation}\label{eq03}
  0 \to K_X(-L_1-L_2) \to K_X \to K_X \otimes \mo_{L_1 \cup L_2} \to 0.
 \end{equation}
 Using \cite[Example $1.8.2$]{lazaI}, the Castelnuovo-Mumford regularity of $\mo_{L_1 \cup L_2}$ is one. Hence, \[H^1(\I_{L_1 \cup L_2}(d-4))=0 \mbox{ for } d \ge 6.\]
 Using the exact sequence 
 \[0 \to \I_X(d-4) \to \I_{L_1 \cup L_2}(d-4) \to \mo_X(-L_1-L_2)(d-4) \to 0,\]
 we conclude that $H^1(\mo_X(-L_1-L_2)(d-4))=0$ for $d \ge 6$. In particular,
 \[H^1(K_X(-L_1-L_2))=H^1(\mo_X(-L_1-L_2)(d-4))=0.\]
 By the short exact sequence \eqref{eq03}, we conclude that 
 \[h^0(K_X)-h^0(K_X(-L_1-L_2))=h^0(K_X \otimes \mo_{L_1 \cup L_2})=h^0(\mo_{L_1 \cup L_2}(d-4))=P_{L_1 \cup L_2}(d-4),\]
 where $P_{L_1 \cup L_2}(t)$ is the Hilbert polynomial of $L_1 \cup L_2$
 and the last equality follows from the fact that the Castelnuovo-Mumford regularity of $\mo_{L_1 \cup L_2}$ is one. 
 Now, the Hilbert polynomial of $L_1 \cup L_2$ is 
 $2(d-4)+1=2d-7$. Therefore, $\codim T_o\NL(\gamma_a) \le 2d-7$. Since $\NL([L_1]) \cap \NL([L_2]) \subset \NL(\gamma_a)$
 and both spaces are of the same dimension, we have $\NL([L_1]) \cap \NL([L_2]) = \NL(\gamma_a)_{\red}$.
 Hence, for a general $u \in \NL(\gamma_a)$, $L_1$ (resp. $L_2$)
 deforms to a line $L_{1,u}$ (resp. $L_{2,u}$) in $\mc{X}_u$ (use 
 the construction of $\mc{L}_1^{\mr{eff}}$ and $\mc{L}_2^{\mr{eff}}$ from the proof of Proposition \ref{prop01}). 
 Then, $\gamma_a$ deforms to $\gamma_{a,u}:=[L_{1,u}]+a[L_{2,u}]$.
 Hence, $\NL(\gamma_a)=\NL(\gamma_{a,u})$. Similarly, as before, we get \[\codim T_u\NL(\gamma_a)= \codim T_u\NL(\gamma_{a,u}) \le 2d-7.\]
 This proves our claim.

 If for $a, a' \ge 2$ with $a \not= a'$, we have $\NL(\gamma_a)=\NL(\gamma_{a'})$, then clearly
 \[\ov{\NL(\gamma_a)}=\ov{\NL(\gamma_{a'})} \subset \ov{\NL([L_1]) \cap \NL([L_2])}.\]
 Since the three spaces have the same dimension and $\ov{\NL([L_1]) \cap \NL([L_2])}$ is reduced (use Proposition \ref{lem4}
 combined with Proposition \ref{prop01}),
 they coincide. But, $\ov{\NL([L_1]) \cap \NL([L_2])}$ is reduced and $\NL(\gamma_a), \NL(\gamma_{a'})$
 are generically non-reduced, as observed earlier. This gives a contradiction, hence $\NL(\gamma_a) \not= \NL(\gamma_{a'})$.
\end{proof}

\begin{rem}
After Theorem \ref{th2}, the remaining parts of Theorem \ref{th1} follows directly from \cite[Theorem $0.2$]{v3}.
\end{rem}



\begin{thebibliography}{10}

\bibitem{ca1}
C.~Ciliberto, J.~Harris, and R.~Miranda.
\newblock General components of the {N}oether-{L}efschetz locus and their
  density in the space of all surfaces.
\newblock {\em Mathematische Annalen}, 282(4):667--680, 1988.

\bibitem{M3}
M.~Green.
\newblock Components of maximal dimension in the {N}oether-{L}efschetz locus.
\newblock {\em J. Differential Geometry}, 29:295--302, 1989.

\bibitem{GH}
P.~Griffiths and J.~Harris.
\newblock Infinitesimal variations of {H}odge structure ({II}): an
  infinitesimal invariant of {H}odge classes.
\newblock {\em Composition Mathematica}, 50(2-3):207--265, 1983.

\bibitem{hoss}
Movasati H.
\newblock Special components of {N}oether-{L}efschetz locus.
\newblock {\em arXiv preprint arXiv:1908.04117}, 2019.

\bibitem{R1}
R.~Hartshorne.
\newblock {\em Algebraic Geometry}.
\newblock Graduate text in Mathematics-52. Springer-Verlag, 1977.

\bibitem{lazaI}
R.~K. Lazarsfeld.
\newblock {\em Positivity in algebraic geometry I: Classical setting: line
  bundles and linear series}, volume~48.
\newblock Springer, 2017.

\bibitem{mat2}
H.~Matsumura.
\newblock {\em Commutative ring theory}, volume~8.
\newblock Cambridge university press, 1989.

\bibitem{S1}
E.~Sernesi.
\newblock {\em Deformaions of Algebraic Schemes}.
\newblock Grundlehren der Mathematischen Wissenschaften-334. Springer-Verlag,
  2006.

\bibitem{v3}
C.~Voisin.
\newblock Composantes de petite codimension du lieu de {N}oether-{L}efschetz.
\newblock {\em Comm. Math. Helve.}, 64(4):515--526, 1989.

\bibitem{voilieu}
C.~Voisin.
\newblock Sur le lieu de {N}oether-{L}efschetz en degr{\'e}s 6 et 7.
\newblock {\em Compositio Mathematica}, 75(1):47--68, 1990.

\bibitem{voicontr}
C.~Voisin.
\newblock Contrexemple {\`a} une conjecture de {J.} {H}arris.
\newblock {\em Comptes rendus de l'Acad{\'e}mie des sciences. S{\'e}rie 1,
  Math{\'e}matique}, 313(10):685--687, 1991.

\bibitem{v4}
C.~Voisin.
\newblock {\em {H}odge Theory and Complex Algebraic Geometry-I}.
\newblock Cambridge studies in advanced mathematics-76. Cambridge University
  press, 2002.

\bibitem{v5}
C.~Voisin.
\newblock {\em {H}odge Theory and Complex Algebraic Geometry-II}.
\newblock Cambridge studies in advanced mathematics-77. Cambridge University
  press, 2003.

\end{thebibliography}
\end{document}